\documentclass[reqno]{amsart}
\usepackage{mathrsfs}
\usepackage{amssymb}
\usepackage{latexsym}
\usepackage{yfonts}
\usepackage{natbib}
\usepackage{amssymb,amsmath,euscript}
\usepackage{color}

\usepackage{graphicx,color}
\usepackage[plainpages=false,colorlinks,hyperindex,bookmarksopen,linkcolor=red,citecolor=blue,urlcolor=blue]{hyperref}
\bibpunct{[}{]}{;}{n}{,}{,}

\newtheorem{tm}{Theorem}
\newtheorem{defin}{Definition}
\newtheorem{rk}{Remark}
\newtheorem{prop}{Proposition}
\newtheorem{lem}{Lemma}
\newtheorem{coro}{Corollary}

\numberwithin{equation}{section}

\newcommand{\RR}[1]{\mathbb{#1}}

\newcommand{\rd}{{\mathbb R^d}}

\newcommand{\rr}{{\mathbb R}}

\def\R{{\mathbb R}}

\def\a{\alpha}

\def\E{{\mathbb E}}
\def\P{{\mathbb P}}

\allowdisplaybreaks

\begin{document}

\title{Intermittence and  time fractional stochastic partial differential equations}

\author{Jebessa B Mijena}
\address{Department of Mathematics, Georgia College \& State University, Milledgeville, GA 31061}
\email{jebessa.mijena@gcsu.edu}
\author{ Erkan Nane}
\address{Department of Mathematics and Statistics, Auburn University, Auburn, AL 36849, USA}
\email{nane@auburn.edu}

\date{\today}

\subjclass[2000]{}

\begin{abstract}
We consider  time fractional  stochastic heat type equation
$$\partial^\beta_tu(t,x)=-\nu(-\Delta)^{\alpha/2} u_t(x)+I^{1-\beta}_t[\sigma(u)\stackrel{\cdot}{W}(t,x)]$$ in $(d+1)$ dimensions, where  $\nu>0$, $\beta\in (0,1)$, $\alpha\in (0,2]$, $d<\min\{2,\beta^{-1}\}\a$,  $\partial^\beta_t$ is the Caputo fractional derivative, $-(-\Delta)^{\alpha/2} $ is the generator of an isotropic stable process,  $\stackrel{\cdot}{W}(t,x)$ is space-time white noise, and $\sigma:\RR{R}\to\RR{R}$ is Lipschitz continuous.
 The time fractional stochastic heat type equations might be used to model phenomenon with random effects with thermal memory. We prove: (i) absolute moments of the solutions of this equation grows exponentially; and (ii) the distances to the origin of the farthest high peaks of those moments grow exactly linearly with time.  These results extend the results of Foondun and Khoshnevisan \cite{foondun-khoshnevisan-09}
and Conus and Khoshnevisan \cite{conus-khoshnevisan}
on the parabolic stochastic heat equations.
\end{abstract}


\maketitle

\section{Introduction}
In recent years a growing literature has been devoted to the study of the time-fractional  diffusion equations. A typical form of the time fractional diffusion equations is $\partial^\beta_tu=\nu \Delta u$  with $\beta\in (0,1)$. These equations are related with anomalous diffusions or diffusions in non-homogeneous media, with random fractal structures; see, for instance, \cite{meerschaert-nane-xiao}.
The Caputo fractional derivative $\partial^\beta_t$  first appeared in \cite{Caputo} is defined for $0<\beta<1$ by
\begin{equation}\label{CaputoDef}
\frac{\partial^\beta u_t(x)}{\partial t^\beta}=\frac{1}{\Gamma(1-\beta)}\int_0^t \frac{\partial
u_r(x)}{\partial r}\frac{dr}{(t-r)^\beta} .
\end{equation}
Its Laplace transform
\begin{equation}\label{CaputolT}
\int_0^\infty e^{-st} \frac{\partial^\beta u_t(x)}{\partial
t^\beta}\,ds=s^\beta \tilde u_s(x)-s^{\beta-1} u_0(x)
\end{equation}
where $\tilde u_s(x) = \int_0^\infty e^{-st}u_t(x)dt$ and incorporates the initial value in the same way as the first
derivative.

Rigorous mathematical approaches  to time fractional diffusion (heat type) equations have been carried out  in  \cite{Koc89, Nig86, Wyss86}; see,  for example,  \cite{NANERW} for a short survey on these results. The solutions to fractional diffusion equations are strictly related with stable densities. Indeed, the stochastic solutions  can be realized through time-change by inverse stable subordinators and therefore we obtain time-changed  processes.   A couple of recent works in this field are \cite{mnv-09, OB09}.

The time fractional SPDEs studied in this paper may arise naturally by considering the heat equation in a material with thermal memory; see \cite{chen-kim-kim-2014} and references therein:
Let $u_t(x), e(t,x)$ and $\stackrel{\to}{F}(t,x)$ denote the body temperature, internal energy and flux density, reps.
Let $\theta, \lambda>0$. The relations
\begin{equation}\begin{split}
\partial_te(t,x)&=-div \stackrel{\to}{F}(t,x)\\
e(t,x)=\theta u_t(x), \ \ \stackrel{\to}{F}(t,x)&=-\lambda\nabla u(t,x)
\end{split}\end{equation}
yields the classical heat equation $\theta \partial_tu_t(x)=\lambda\Delta u_t(x)$.

According to the law of classical heat equation, the speed of heat flow is infinite.
But  the propagation speed can be finite because the heat flow can be disrupted by the response of the material.
In a material with thermal memory Lunardi and Sinestrari \cite{lunardi-sinestrari},  von Wolfersdorf \cite{von-wolfersdorf}  showed that
$$e(t,x)=\bar{\theta}u_t(x)+\int_0^tn(t-s)u_s(x)ds$$
holds with some appropriate constant $\bar{\beta}$ and kernel $n$.  In most cases we would have $n(t)=\Gamma(1-\beta_1)^{-1}t^{-\beta_1}$.
The convolution implies that the nearer past affects the present more.
If in addition the internal energy also depends on past random effects, then
\begin{equation}\label{phys-1}
\begin{split}
e(t,x)&=\bar{\theta}u_t(x)+\int_0^tn(t-s)u_s(x)ds\\
\ \ \ &+\int_0^t l(t-s)h(s, u_s(x))\frac{\partial {W}(ds, x)}{\partial x}
\end{split}\end{equation}
where $W$ is the space time white noise, modeling the random effects.
Take $l(t)=\Gamma(2-\beta_2)^{-1}t^{1-\beta_2}$,
then after differentiation \eqref{phys-1} gives
\begin{equation}\label{eq:physics-motivation}
\partial_t^{\beta_1}u_t(x)=div\stackrel{\to}{F}+ \frac{1}{\Gamma(1-\beta_2)} \int _0^t(t-s)^{-\beta_2}h(s,u_s(x))\frac{\partial W(ds,x)}{\partial x}.
\end{equation}

Let $\gamma>0$, define the fractional integral by
$$I^{\gamma}_tf(t):=\frac{1}{\Gamma(\gamma)} \int _0^t(t-\tau)^{\gamma-1}f(\tau)d\tau.$$
For every  $\gamma>0$, and   $g\in L^\infty(\rr_+)$ or $g\in C(\rr_+)$, we have the following relation
$$ \partial _t^\gamma I^\gamma_t g(t)=g(t).$$

 In this paper we will study intermittency and intermittency fronts for the solution of  the type of stochastic equations in \eqref{eq:physics-motivation} and  its extensions:
 \begin{equation}\label{tfspde}
\begin{split}
 &\partial^\beta_tu_t(x)=-\nu(-\Delta)^{\alpha/2} u_t(x)+I^{1-\beta}_t[\sigma(u)\stackrel{\cdot}{W}(t,x)],\ \ t> 0,\, x\in\R^d;\\
 &u_t(x)|_{t=0}=u_0(x),
 \end{split}
 \end{equation}
where the initial datum $u_0$ is $L^p(\Omega)$-bounded ($p\ge 2$), that is,
\begin{equation}\label{Eq:Initial}
\sup_{x\in\R^d}\E[|u_0(x)|^p]<\infty,
\end{equation}
 $-(-\Delta)^{\alpha/2} $ is the fractional Laplacian with $\alpha\in (0,2]$, and $\stackrel{\cdot}{W}(t,x)$ is a  space-time white noise with $x\in \R^d$.
When $\sigma(u)=1$, the fractional integral above in equation \eqref{tfspde}  is defined as
 $$I^{1-\beta}_t[\sigma(u)\stackrel{\cdot}{W}(t,x)]= \frac{1}{\Gamma(1-\beta)} \int _0^t(t-\tau)^{-\beta}\frac{\partial W(d\tau,x)}{\partial x}$$  is well defined only when $0<\beta<1/2$. It is a type of Rieman-Liouville process.

Let $G_t(x)$ be the fundamental solution
to the fractional heat type equation
\begin{equation}\label{Eq:Green0}
\partial^\beta_t G_t(x)=-\nu(-\Delta)^{\alpha/2}G_t(x).
\end{equation}
We know that $G_t(x)$ is the density function of $X(E_t)$, where $X$ is an isotropic $\a$-stable L\'evy process in $\R^d$ and $E_t=\inf \{u:\ D(u)>t\}$,  is the first passage time of a $\beta$-stable subordinator $D=\{D_r,\,r\ge0\}$, or the inverse stable subordinator of index $\beta$: see, for example, Bertoin \cite{bertoin} for properties of these processes, Baeumer and Meerschaert \cite{fracCauchy} for more on time fractional diffusion equations, and Meerschaert and Scheffler  \cite{limitCTRW} for properties of the inverse stable subordinator $E_t$.

Let $p_{{X(s)}}(x)$ and $f_{E_t}(s)$ be the density of $X(s)$ and $E_t$, respectively. Then the Fourier transform of $p_{{X(s)}}(x)$ is given by
\begin{equation}\label{Eq:F_pX}
\widehat{p_{X(s)}}(\xi)=e^{-s\nu|\xi|^\a},
\end{equation}
and
\begin{equation}\label{Etdens0}
f_t(x)=t\beta^{-1}x^{-1-1/\beta}g_\beta(tx^{-1/\beta}),
\end{equation}
where $g_\beta(\cdot)$ is the density function of $D_1.$ The function $g_\beta(u)$ [cf. Meerschaert and Straka (2013)] is infinitely differentiable on the entire real line, with $g_\beta(u)=0$ for $u\le 0$.

By conditioning, we have
\begin{equation}\label{Eq:Green1}
G_t(x)=\int_{0}^\infty p_{_{X(s)}}(x) f_{E_t}(s)ds.
\end{equation}

 A related  time-fractional  SPDE was studied by Chen  et al. \cite{chen-kim-kim-2014}. They have proved existence, uniqueness and regularity of the solutions to the time-fractional parabolic type SPDEs using cylindrical Brownian motion in Banach spaces, in line with the methods in \cite{daPrato-Zabczyk}.
The existence and uniqueness of the solution to (\ref{tfspde}) has been studied by Nane et al \cite{nane-et-al-2014} under Global Lipchitz conditions on $\sigma$, using the White noise  approach of \cite{walsh}:
We say that
an $\mathcal{F}_t$-adapted random field $\{u(t,x),\,t\ge 0,\,x\in\R^d\}$ is said to be a mild solution of (\ref{tfspde}) with initial value $u_0$ if the following integral equation is fulfilled
\begin{equation}\label{Eq:Mild}
u_t(x)=\int_{\R^d}u_0(y)G_t(x-y)dy+\int_0^t\int_{\R^d}\sigma(u_r(y))G_{t-r}(x-y)W(drdy).
\end{equation}

Let $T$ be a fixed positive number, and let $B_{T,p}$ denote the family of all $\mathcal{F}_t$-adapted random fields $\{u_t(x),\,t\in [0,\,T],\,x\in\R^d\}$ satisfying
\begin{equation}\label{Eq:SquareInt}
\sup_{x\in\R^d}\sup_{t\in[0,\,T]}\E\left[|u_t(x)|^p\right]<\infty,
\end{equation}
with the convention that $B_{T,2}=B_T$. It is easy to check that for each fixed $T$ and $p,$ $B_{T,p}$ is a Banach space.

Nane et al. \cite{nane-et-al-2014} proved the existence and uniqueness result  for the equation \eqref{tfspde} when
 $d<\min\{2,\beta^{-1}\}\a$,  equation (\ref{tfspde}) subject to (\ref{Eq:Initial}) and global Lipschitz condition on $\sigma$ has an a.s.-unique solution $u_t(x)$ that satisfies that for all $T>0,$
$u_t(x)\in B_{T,p}.$
For a comparison of the two approaches to SPDE's see the paper by Dalang and Quer-Sardanyons \cite{Dalang-Quer-Sardanyons}.


In this paper we  study the intermittency behavior of the  solution of the time fractional spde \eqref{tfspde}. We adopt the definition given in \cite[Chapter 7]{khoshnevisan-cbms}:
 The random field $u_t(x)$
is called  weakly intermittent if $\inf_{z\in \rd}|u_0(z)|>0$, and
$\eta_k(x)/k$ is strictly increasing for $k\geq 2$ for all $x\in \rd$, where
\begin{equation}\label{lyapunov-x}
\eta_k(x):=\liminf_{t\to\infty}\frac1t\log\E(|u_t(x)|^k).
\end{equation}
 There is a huge literature on the study of intermittency of SPDEs, see, for example, \cite{foondun-khoshnevisan-09, khoshnevisan-cbms} and the reference therein.

Next we consider the solution $u$ to the time fractional stochastic heat equation \eqref{tfspde} when $d=1$ and $\alpha=2$.  We extend the  results  on the stochastic heat equation corresponding to $\beta =1$ in \cite{conus-khoshnevisan, khoshnevisan-cbms} to the time fractional stochastic heat equation. Assume that $\inf_{x\in \RR{R}}u_0(x)=0$. To be more precise,  choose a measurable initial function $u_0:\RR{R}\to\RR{R}_+$ that is bounded, has compact support, and is strictly positive in an open subinterval of $(0,\infty)$. We also assume that $\sigma$  satisfies
 $\sigma(0)=0.$

 It turns out that the preceding assumptions imply that the solution develops tall peaks over time which means that $t\to \sup_{x\in \RR{R}} \E{|u_t(x)|^2}$ grows exponentially rapidly with $t$. There appears another phenomena called intermittency fronts that the distances of the farthest peaks of the moments of the solution to \eqref{tfspde} grow  linearly with time as $\theta t$: if $\theta$ is sufficiently small, then the quantity $\sup_{|x|>\theta t} \E{|u_t(x)|^2}$ grows exponentially quickly as $t\to\infty$; whereas the preceding quantity vanishes exponentially rapidly if $\theta$ is sufficiently large.
Thus, it makes sense to consider, for every $\theta\geq 0$,
\begin{equation}\label{intermittecy-front-L}
\mathscr{L}(\theta):= \limsup_{t\rightarrow \infty}\frac{1}{t}\sup_{|x|>\theta t}\log \E\left(|u_t(x)|^2\right).
\end{equation}
We can think of $\theta_L>0$ as an \textit{intermittency lower front} if $\mathscr{L}(\theta)<0$ for all $\theta > \theta_L,$ and of $\theta_U>0$ as an \textit{intermittency upper front} if $\mathscr{L}(\theta)>0$ whenever $\theta<\theta_U.$

We obtain  bounds for $\theta_L$ and $\theta_U$ that extends the results of \cite{conus-khoshnevisan} to the case of time fractional SPDEs with crucial nontrivial changes to the methods in \cite{conus-khoshnevisan, khoshnevisan-cbms}.

Next we want to give an outline of the paper. We give some preliminary results in section 2. The intermittency for the solution of  \eqref{tfspde} will be proved in section 3, the main result is  Theorem \ref{intermittency-thm}. In section 4 we prove the  bounds for the intermittency fronts  the solution of  \eqref{tfspde}. The main result is Theorem  \ref{maintheoremtwo}.

\section{Preliminaries}

Applying the Laplace transform with respect to time variable t, Fourier transform with respect to space variable x.
 Laplace-Fourier transform of $G$ is given by
\begin{eqnarray}
\int_{0}^{\infty}\int_{{\R^d}} e^{-\lambda t + i\xi\cdot x}G_t(x)dxdt&=&\int_{0}^{\infty}e^{-\lambda t}dt\int_{0}^\infty f_{E_t}(s)ds\int_{{\R^d}} e^{i\xi\cdot x}p_{X(s)}(x)dx\nonumber\\
&=&\int_{0}^{\infty}e^{-s\nu|\xi|^\alpha}ds\int_{0}^\infty e^{-\lambda t}f_{E_t}(s)dt\nonumber\\
&=&\frac{\lambda^\beta}{\lambda}\int_{0}^\infty e^{-s(\nu|\xi|^\alpha +\lambda^\beta) }ds\nonumber\\
&=&\frac{\lambda^{\beta-1}}{\lambda^{\beta} + \nu|\xi|^\alpha}.
\end{eqnarray}
here we used the fact that the laplace transform $t\to\lambda $ of $f_{E_t}(u)$ is given by $\lambda^{\beta-1}e^{-u\lambda^\beta}$.
Using the convention, $\sim$ to denote the Laplace transform and $\ast$ the Fourier transform we get
\begin{equation}
\tilde{G}^\ast_t(x) = \frac{\lambda^{\beta-1}}{\lambda^{\beta} +\nu |\xi|^\alpha}.
\end{equation}
Inverting the Laplace transform, it yields
\begin{equation}\label{fouriertransformofG}
G^\ast_t(\xi) = E_\beta(-\nu|\xi|^\alpha t^\beta),
\end{equation}
where
\begin{equation}\label{ML-function}
E_\beta(x) = \sum_{k=0}^\infty\frac{x^k}{\Gamma(1+\beta k)}
\end{equation}
 is the Mittag-Leffler function. In order to invert the Fourier transform, we will make use of the integral \cite[eq. 12.9]{haubold-mathai-saxena}
\begin{equation}
\int_0^\infty\cos(ks)E_{\beta,\a}(-as^\mu)ds = \frac{\pi}{k}H_{3,3}^{2,1}\bigg[\frac{k^\mu}{a}\bigg|^{(1,1), (\a,\beta), (1,\mu/2)}_{(1,\mu),(1,1),(1,\mu/2)}\bigg],\nonumber
\end{equation}
where $\mathcal{R}(\a)>0,\mathcal{\beta}>0,k>0,a>0, H_{p,q}^{m,n}$ is the H-function given in \cite[Definition 1.9.1, p. 55]{mathai} and the formula
\begin{equation}
\frac{1}{2\pi}\int_{-\infty}^\infty e^{-i\xi x}f(\xi)d\xi = \frac{1}{\pi}\int_0^\infty f(\xi)\cos(\xi x)d\xi.\nonumber
\end{equation}
Then this gives the function as
\begin{equation}\label{G-function}
G_t(x) = \frac{1}{|x|} H_{3,3}^{2,1}\bigg[\frac{|x|^\a}{\nu t^\beta}\bigg|^{(1,1), (1,\beta), (1,\a/2)}_{(1,\a),(1,1),(1,\a/2)}\bigg].
\end{equation}
Note that for $\a = 2$ using reduction formula for the H-function we have
\begin{equation}
G_t(x) = \frac{1}{|x|}H^{1,0}_{1,1}\bigg[\frac{|x|^2}{\nu t^\beta}\bigg|^{(1,\beta)}_{(1,2)}\bigg]
\end{equation}
Note  that for $\beta = 1$ it reduces to the Gaussian density
\begin{equation}
G_t(x) = \frac{1}{(4\nu\pi t)^{1/2}}\exp\left(-\frac{|x|^2}{4\nu t}\right).
\end{equation}
Recall
 uniform estimate of Mittag-Leffler function \cite[Theorem 4]{simon}
 \begin{equation}\label{uniformbound}
 \frac{1}{1 + \Gamma(1-\beta)x}\leq E_{\beta}(-x)\leq \frac{1}{1+\Gamma(1+\beta)^{-1}x} \ \  \ \text{for}\ x>0.
 \end{equation}
\begin{lem}\label{Lem:Green1} For $d < 2\a,$
\begin{equation}\label{Eq:Greenint}
\int_{{\R^d}}G^2_t(x)dx  =C^\ast t^{-\beta d/\a}
\end{equation}
where $C^\ast = \frac{(\nu )^{-d/\a}2\pi^{d/2}}{\a\Gamma(\frac d2)}\frac{1}{(2\pi)^d}\int_0^\infty z^{d/\a-1} (E_\beta(-z))^2 dz.$
\end{lem}

\begin{proof}
Using Plancherel theorem and \eqref{fouriertransformofG}, we have

\begin{eqnarray}
\int_{\rd}|G_t(x)|^2 dx &=& \frac{1}{(2\pi)^d}\int_{\rd}|G^\ast_t(\xi)|^2 d\xi= \frac{1}{(2\pi)^d}\int_{\rd}|E_\beta(-\nu|\xi|^\a t^\beta)|^2 d\xi\nonumber\\
&=&\frac{2\pi^{d/2}}{\Gamma(\frac d2)}\frac{1}{(2\pi)^d}\int_0^\infty r^{d-1} (E_\beta(-\nu r^\a t^\beta))^2 dr.\label{square-kernel}\\
&=&\frac{(\nu t^\beta)^{-d/\a}2\pi^{d/2}}{\a\Gamma(\frac d2)}\frac{1}{(2\pi)^d}\int_0^\infty z^{d/\a-1} (E_\beta(-z))^2 dz.
\end{eqnarray}
We used the integration in  polar coordinates for radially symmetric function in  the last equation above. Now using equation \eqref{uniformbound} we get
\begin{eqnarray}\label{uniformboundforE}
\int_{0}^\infty\frac{z^{d/\a-1} }{(1+\Gamma(1-\beta)z)^2} dr&\leq&\int_0^\infty z^{d/\a-1} (E_\beta(-z))^2 dz\nonumber\\
&\leq&\int_{0}^\infty\frac{z^{d/\a-1} }{(1+\Gamma(1+\beta)^{-1}z)^2}dz
\end{eqnarray}
Hence  $\int_0^\infty z^{d/\a-1} (E_\beta(-z))^2 dz<\infty$ if and only if $d<2\a$. In this case
 the upper bound in equation \eqref{uniformboundforE} is
$$\int_0^\infty \frac{z^{d/\a-1} }{(1+\Gamma(1+\beta)^{-1}z)^2} dz = \frac{\text{B}(d/\a, 2-d/\a)}{\Gamma(1+\beta)^{-d/\a}},$$
where $\text{B}(d/\a, 2-d/\a)$ is a Beta function.
\end{proof}
\begin{rk}
For  special case $d = 1, \a = 2$ and $\beta = 1$ we get
\begin{equation}\label{specialcase}
\int_{-\infty}^{\infty}G_t(x)^2 dx = \frac{1}{(8\nu\pi t)^{1/2}}.
\end{equation}
\end{rk}
\begin{lem} For $\lambda\in\rd$ and $\a = 2$,
$$\int_{\rd}e^{\lambda\cdot x}G_s(x)dx = E_\beta(\nu |\lambda|^2 s^\beta).$$
\end{lem}
\begin{proof}
Using uniqueness of Laplace transform we can easily show $\E[E_s^k] = \frac{\Gamma(1+k)s^{\beta k}}{\Gamma(1+\beta k)}$ for $k>-1$. Therefore,
using this and moment-generating function of Gaussian densities we have
\begin{eqnarray}
\int_{\rd}e^{\lambda\cdot x}G_s(x)dx &=& \int_{0}^\infty \int_{\rd}e^{\lambda\cdot x}\frac{e^{-\frac{|x|^2}{4\nu u}}}{(4\pi\nu u)^{d/2}}dxf_{E_s}(u)du\nonumber\\
&=& \int_{0}^\infty e^{\nu |\lambda|^2 u}f_{E_s}(u)du\nonumber\\
&=&\sum_{k=0}^{\infty}\frac{\nu^k|\lambda|^{2k}}{k!}\int_0^\infty u^kf_{E_s}(u)du\nonumber\\
&=& \sum_{k=0}^{\infty}\frac{\nu^k|\lambda|^{2k}}{k!}\frac{\Gamma(1+k)s^{\beta k}}{\Gamma(1+\beta k)}.
\end{eqnarray}

\end{proof}

$\stackrel{\cdot}{W}(t,x)$ is a space-time white noise  with $x\in \R^d$, which is assumed to be adapted with respect to a filtered probability space $(\Omega, \mathcal{F},  \mathcal{F}_t, \P)$, where $ \mathcal{F}$ is complete and the filtration $\{\mathcal{F}_t, t\geq 0\}$ is right continuous. $\stackrel{\cdot}{W}(t,x)$ is  generalized processes with covariance given by
$$
\E\bigg[\stackrel{\cdot}{W}(t,x) \stackrel{\cdot}{W}(s,y)\bigg]=\delta(x-y)\delta(t-s).
$$
That is, $W(f) $ is a random field indexed by functions $ f\in L^2((0,\infty)\times \R^d )$ and for all $f,g\in L^2((0,\infty)\times \R ^d)$, we have
$$
\E\bigg[W(f)W(g)\bigg]=\int_0^\infty \int_{\R^d} f(t,x)g(t,x)dxdt.
$$

Hence $W(f)$ can be represented as
$$
W(f)=\int_0^\infty \int_{\R^d} f(t,x)W(dxdt).
$$
Note that $W(f)$ is $\mathcal{F}_t$-measurable whenever $f$ is supported on $[0,t]\times\R^d$.

Let $\Phi$ be a random field, and for every $\gamma>0$ and $k\in [2, \infty)$ define
\begin{equation}
\mathcal{N}_{\gamma,k}(\Phi) := \sup_{t\geq 0}\sup_{x\in\R^d}\left(e^{-\gamma t}||\Phi_t(x)||_k\right):= \sup_{t\geq 0}\sup_{x\in\R^d}\left(e^{-\gamma t}\bigg[\E|\Phi_t(x)|^k\bigg]^{1/k}\right).
\end{equation}
If we identify a.s.-equal random fields, then every $\mathcal{N}_{\gamma, k} $ becomes a norm. Moreover, $\mathcal{N}_{\gamma, k} $ and $\mathcal{N}_{\gamma', k} $ are equivalent norms for all $\gamma, \gamma'>0$ and $k\in [2,\infty).$ Finally, we note that if $\mathcal{N}_{\gamma, k}(\Phi) <\infty$ for some $\gamma >0$ and $k\in[2,\infty)$, then $\mathcal{N}_{\gamma, 2}<\infty$  as well, thanks to Jensen's inequality.

\begin{defin}
We denote by $\mathcal{L}^{\gamma, 2}$  the completion of the space of all simple  random fields in the norm $\mathcal{N}_{\gamma, 2}.$
\end{defin}

Given a random field $\Phi := \{\Phi_t(x)\}_{t\geq 0, x\in\R^d}$ and space-time noise $W$, we define the [space-time] \textit{stochastic convolution} $G\circledast \Phi$ to be the random field that is defined as
$$(G\circledast \Phi)_t(x) := \int_{(0,t)\times \R^d} G_{t-s}(y-x)\Phi_s(y)W(dsdy),$$
for $t>0$ and $x\in\R^d,$ and $(G\circledast W)_0(x) := 0.$

Define
\begin{equation}
G_s^{(t,x)}(y) := G_{t-s}(y-x)\cdot {\bf{1}}_{(0,t)}(s)\ \ \ \mbox{for all}\ s\geq 0\ \mbox{and}\ y\in\R^d.
\end{equation}
Clearly, $G^{(t,x)}\in L^2(\R_+\times \R^d)$ for $\beta d/\a<1$; in fact,
$$\int_0^\infty ds \int_{\R^d} [G_s^{(t,x)}(y)]^2dy =  \int_0^t ds \int_{\R^d} [G_s(y)]^2dy = C^\ast t^{1-\beta d/\a}<\infty.$$
This computation follows from Lemma \ref{Lem:Green1}. Thus, we may interpret the random variable $(G\circledast \Phi)_t(x)$ as the stochastic integral $\int G_s^{(t,x)}\Phi dW$, provided that $\Phi$ is in $\mathcal{L}^{\gamma, 2}$ for some $\gamma>0.$ Let us recall that $\Phi\mapsto G\circledast \Phi$ is a random linear map; that is, if $\Phi, \Psi
\in \mathcal{L}^{\gamma, 2}$ for some $\gamma > 0,$ then for all $a, b\in\R$ the following holds almost surely:
\begin{eqnarray}
&&\int_{(0,t)\times \R^d} G_{t-s}(y-x)(a\Phi_s(y) + b\Psi(y))W(dsdy)\nonumber\\
&=&a\int_{(0,t)\times \R^d} G_{t-s}(y-x)\Phi_s(y)W(dsdy) + b\int_{(0,t)\times \R^d} G_{t-s}(y-x)\Psi_s(y)W(dsdy)\nonumber
\end{eqnarray}
\section{Intermittency}\label{intermittency}

To motivate the study of intermittency we include a part of the discussion in Khoshnevisan \cite[Section 7.1]{khoshnevisan-cbms}.
 Let $\{\psi_t(x)\}_{t\geq 0, x\in\rd}$ be a non-negative random field that is stationary in the $x$ parameter. We will refer to the function
 $$
 \eta(k):=\lim_{t\to\infty} \frac1t\log\bigg(\E[\psi_t(x)]^k\bigg)
 $$
 as the Lyapunov exponent of $\psi$ provided that $\eta(k)<\infty$ for all real numbers $k\geq 1$.

 For the sake of simplicity, we assume also that
 $$
 \E[\psi_t(x)]=1\ \ \mathrm{for\ all}\ t\geq 0\ \mathrm{and}\ x\in \rd.
 $$

 The following definition is more or less standard terminology; see Khoshnevisan \cite{khoshnevisan-cbms} and references therein:
 We say that the random field $\psi$ is intermittent when its Lyapunov exponent $\eta$ has the property that $k\to k^{-1}\eta(k) $ is strictly increasing on $[2,\infty)$

  Jensen's inequality ensures that the function $k\to(\E[\psi_t(x)]^k)^{1/k}=||\psi_t(x)||_k $ is always nondecreasing, this implies that $k\to k^{-1}\eta(k) $ is also nondecreasing. Thus, $\psi$ is intermittent if and only if ``nondecreasing'' is replaced  by ``strictly increasing''.

  The following observation of Carmona and Molchanov \cite[Theorem 3.1.2]{carmona-molchanov} gives a sufficient condition for intermittency: see \cite[Proposition 7.2]{khoshnevisan-cbms} for a proof of the next proposition.

  \begin{prop}
  \label{intermittence-sufficient}
  If $\eta(k)<\infty$ for all sufficiently  large $k$, then the function $\eta$ is well-defined and convex on $(0,\infty)$. Moreover, If $\eta(k_0)>0$ for some $k_0>1$, then $k\to k^{-1}\eta(k) $ is strictly increasing on $[k_0,\infty)$
  \end{prop}

    Proposition \ref{intermittence-sufficient} says that we need to show $\eta(2)>0$ in order to prove that $\psi $ is intermittent.

    The discussion in Khoshnevisan \cite[Section 7.1]{khoshnevisan-cbms}  gives an implication of the intermittency for the random field $\psi$:
    There is a non-random strictly increasing, strictly positive sequence $\{\theta_j\}_{j=1}^\infty$ such that
    \begin{eqnarray}
    &0<\limsup_{N\to\infty}\frac1N \max_{1\leq l\leq \exp(\theta_i N)}\log \psi_N(l)\nonumber\\
    \ \ \ \ \ \ \ \ \ \ \ \ \ \ \ \ \ \ \ \ \ \ \ \ \ \ \ &<\liminf_{N\to\infty}\frac1N \max_{1\leq l\leq \exp(\theta_{i} N)}\log \psi_N(l)<\infty
    \end{eqnarray}
    almost surely for every   $i\geq 1$. Hence, when $N$ is large, $x\to\psi_N(x)$ experiences increasingly-large peaks--on an exponential scale with $N$-- as $x$ grows on different scales with $N$.

\subsection{Intermittency and time fractional stochastic heat equation}

Recall the definition of $\eta_k(x)$ from \eqref{lyapunov-x}.  One can prove the following fact using the method of proof of Proposition \ref{intermittence-sufficient}: If $\eta_2(x)>0$ for all $x\in \rd$, then $\eta_k(x)/k$ is strictly increasing for $k\geq 2$ for all $x\in \rd$.

\begin{tm}\label{intermittency-thm}
Let  $d<\min\{2,\beta^{-1}\}\a$.
If $\inf_{z\in \rd}|u_0(z)|>0$, then $$\inf_{x\in \rd}\eta_2(x)\geq [C^\ast(L_\sigma)^2\Gamma(1-\beta d/\a)]^{\frac{1}{(1-\beta d/\a)}}$$
where
\begin{equation}\label{conecondition}
L_\sigma:=\inf_{z\in \rd}|\sigma(z)/z|.
\end{equation}
Therefore, the solution  $u_t(x)$ of \eqref{tfspde} is weakly intermittent when $\inf_{z\in \rd}|u_0(z)|>0$ and $L_\sigma>0$.
\end{tm}



This theorem extends the results of \cite{foondun-khoshnevisan-09} to the time fractional stochastic heat type equations.
\begin{rk}
Recall the constant $C^*=const\cdot \nu^{-d/\a}$. Hence Theorem \ref{intermittency-thm} implies the so-called ``very fast dynamo property,'' $\lim_{\nu\to\infty}\inf_{x\in \rd}\eta_2(x)=\infty$. This property has been studied in fluid dynamics \cite{arponen-horvai-07, baxendale-rozovskii-93, galloway-03}.
\end{rk}
Theorem \ref{intermittency-thm} is proved by using  an  application of  non-linear renewal theory. Hence, we discuss such material briefly first, we will return to the proof of Theorem \ref{intermittency-thm} after this discussion.
\subsection{Renewal theory}

Consider the renewal equation:
\begin{equation}\label{renewaleqn}
f(t) = a(t) + \int_0^tf(s)g(t-s)ds \ \ \ (t>0).
\end{equation}
where $a :\R_+\rightarrow \R_+$ is measurable and non-decreasing, and
\begin{equation}\label{thegfunction}
g(\tau) := b/\tau^{\theta}\ \ \ \mbox{for all}\ \tau>0,
\end{equation}
for a positive and finite constant $b$ and $0<\theta < 1$. Of course, the function $f$ denotes the solution to the renewal equation \eqref{renewaleqn}, if indeed a solution exists.

For every measurable function $h :(0,\infty)\rightarrow \R_+,$ define the "titled version" $\tilde{h}$ of $h$ as
$$\tilde{h}(\tau) := e^{-c\tau}h(\tau)\ \ \ (\tau>0),$$
where $c = \left(b\Gamma(1-\theta)\right)^{1/(1-\theta)}.$

The key property of tilting is that $\tilde{g}$ is a probability density function on $(0,\infty)$, where $g$ is specifically the function defined in \eqref{thegfunction}. Furthermore, $f$ is a solution to \eqref{renewaleqn} if and only if $\tilde{f}$ solves the renewal equation
\begin{equation}\label{tilderenewaleqn}
\tilde{f}(t) = \tilde{a}(t) + \int_0^t\tilde{f}(t-s)\tilde{g}(s)ds\ \ \ (t>0).
\end{equation}

Since $\tilde{a}$ is a bounded measurable function and $\tilde{g}$ is a probability density function, classical renewal theory \cite[Chapter 9]{feller} tells us that \eqref{tilderenewaleqn} has a unique non-negative bounded solution $\tilde{f}$. Consequently, \eqref{renewaleqn} has a unique non-negative solution $f$ that grows at most exponentially; in fact, $f(t) = O(\exp(ct))$ as $t\rightarrow \infty.$ Finally, note that, because $\tilde{a}$ decreasing, hence "directly integrable," we may apply the renewal theorem-see for example Feller \cite[p. 363]{feller}-and deduce that
$$\lim_{t\rightarrow\infty}\tilde{f}(t) = \frac{\int_0^\infty\tilde{a}(y)dy}{\int_0^\infty y\tilde{g}(y)dy}=\frac{c}{1-\theta}\int_0^\infty a(y)e^{-cy}dy,$$
since $\int_0^\infty y\tilde{g}(y)dy = (1-\theta)/c.$ We can  summarize our findings as an  elementary theorem.
\begin{tm}\label{theorem2}
In the preceding setup, \eqref{renewaleqn} admits a unique solution $f$ subject to $\limsup_{t\rightarrow\infty}t^{-1}\log f(t) < \infty.$ Moreover,
$$\lim_{t\rightarrow\infty}e^{-ct}f(t) = \frac{c}{1-\theta}\int_0^\infty a(y)e^{-cy}dy,$$
where $c = \left(b\Gamma(1-\theta)\right)^{1/(1-\theta)}.$
\end{tm}
\begin{defin}
We say that $h$ is a subsolution to \eqref{renewaleqn} if $h:(0,\infty)\rightarrow\R_+$ is measurable and $h(t)\leq a(t) + \int_0^t h(s)g(t-s)ds$ for all $t>0.$ We say that $h$ is a supersolution to \eqref{renewaleqn} if in addition to measurability $h$ satisfies $h(t)\geq a(t) + \int_0^th(s)g(t-s)ds$ for all $t>0.$
\end{defin}
Note, in particular, that if $h:(0,\infty)\rightarrow \R_+$ is a non-random  function then $\mathcal{N}_{\gamma,2}(h) = \sup_{t>0}(e^{-\gamma t}h(t))$ for every $\gamma > 0.$ Then we have the following comparison theorem for renewal equations.
\begin{tm}\label{sup-sub-solution}
Suppose $f$ solves \eqref{renewaleqn} and $F$ is a non-negative super solution to \eqref{renewaleqn} that satisfies $\mathcal{N}_{\gamma,2}(F)<\infty$ for some $\gamma > 0$. Then $F(t)\geq f(t)$ for all $t>0.$ Similarly, is $H$ is a subsolution to \eqref{renewaleqn} that satisfies
 $\mathcal{N}_{\gamma,2}(H)<\infty$ for some $\gamma > 0,$ then $H(t)\leq f(t)$ for all $t>0.$
 \end{tm}

\begin{proof}
Our proof uses the   Picard's iteration adapted from Georgiou et al \cite[Appendix]{georgiou-et-al} with crucial changes to their methods.

Let $f^{(0)}(t):= F(t),$ and define iteratively
$$f^{(n+1)}(t) := a(t) + \int_{0}^{t}f^{(n)}(s)g(t-s)ds\ \ (t>0).$$
We may observe that
$$f^{(1)}(t) = a(t) + \int_0^tF(s)g(t-s)ds \leq F(t) = f^{(0)}(t)\ \ \  (t>0).$$
And if $f^{(k)}\leq f^{(k-1)}$ for some integer $k>0,$ then
$$f^{(k+1)}(t)\leq a(t) + \int_0^tf^{(k-1)}(s)g(t-s)ds = f^{(k)}(t)\ \ \ (t>0).$$
Therefore, it follows from induction that
\begin{equation}\label{supersolution}
F = f^{(0)}\geq f^{(1)}\geq f^{(2)}\geq \cdots.
\end{equation}
Clearly,
\begin{eqnarray}
\left(f^{(n)}\ast g\right)(t) &:=& \int_0^tf^{(n)}(s)g(t-s)ds\nonumber\\
&\leq& \mathcal{N}_{\gamma, 2} (f^{(n)}) \int_0^te^{\gamma s}g(t-s)ds\nonumber\\
&=&e^{\gamma t}\mathcal{N}_{\gamma, 2}(f^{(n)}) \int_0^te^{-\gamma r}g(r)dr\nonumber\\
&\leq&e^{\gamma t}\mathcal{N}_{\gamma, 2}(f^{(n)}) \int_0^\infty e^{-\gamma r}g(r)dr\nonumber\\
&\leq& \frac{e^{\gamma t}}{2}\mathcal{N}_{\gamma, 2}(f^{(n)}),\nonumber
\end{eqnarray}
for every $\gamma\geq 2^{1-\theta}b.$ Therefore,
$$\mathcal{N}_{\gamma, 2}(f^{(n)}\ast g)\leq \frac{1}{2}\mathcal{N}_{\gamma, 2}(f^{(n)})\ \ \ (\gamma \geq 2^{1-\theta}b, n>0),$$
and because $a(t)\leq a(0)$ for all $t\geq 0,$
$$\mathcal{N}_{\gamma, 2}(f^{(n+1)})\leq a(0) + \frac{1}{2}\mathcal{N}_{\gamma, 2}(f^{(n)}) \ \ \ (\gamma \geq 2^{1-\theta}b, n>0).$$
We apply the preceding repeatedly to see that
$$\mathcal{N}_{\gamma, 2}(f^{(n+1)})\leq a(0)\sum_{j=0}^{n}2^{-j} + 2^{-(n+1)}\mathcal{N}_{\gamma, 2}(F) \ \ \ (\gamma \geq 2^{1-\theta}b, n>0).$$
Since $\gamma\mapsto \mathcal{N}_{\gamma, 2}(F)$ is decreasing, we may choose $\gamma \geq 2^{1-\theta}b$ large enough to ensure that $\mathcal{N}_{\gamma, 2}(F) < \infty.$ For this particular choice of $\gamma,$
$$\mathcal{N}_{\gamma, 2}(f^{(n+1)})\leq 2a(0) + 2^{-(n+1)}\mathcal{N}_{\gamma, 2}(F) \leq 2a(0) + \mathcal{N}_{\gamma, 2}(F)<\infty,$$
uniformly for all $n\geq 0.$

Similarly, and for the same choice of $\gamma,$
\begin{eqnarray}
\left|f^{(n+1)}(t) - f^{(n)}(t)\right|&\leq& \int_0^t|f^{(n)}(s) - f^{(n-1)}(s)|g(t-s)ds\nonumber\\
&\leq&e^{\gamma t}\mathcal{N}_{\gamma, 2}(f^{(n)} - f^{(n-1)})\int_0^\infty e^{-(t-s)\gamma}g(t-s)ds\nonumber\\
&=&\frac{e^{\gamma t}c\Gamma(1-\theta)}{\gamma^{(1-\theta)}}\mathcal{N}_{\gamma, 2}(f^{(n)} - f^{(n-1)})\nonumber\\
&\leq&\frac{e^{\gamma t}}{2}\mathcal{N}_{\gamma, 2}(f^{(n)} - f^{(n-1)})\ \ \ \ (n\geq 0).\nonumber
\end{eqnarray}
It follows that
$$\mathcal{N}_{\gamma, 2}(f^{(n+1)} - f^{(n)})\leq\frac{1}{2}\mathcal{N}_{\gamma, 2}(f^{(n)} - f^{(n-1)})\ \ \ (n\geq 0),$$
whence the $f^{(n)}$'s converge-in the norm $\mathcal{N}_{\gamma, 2}$-to some function $h$ (as $n\rightarrow\infty).$ Moreover, $h(t) :=\lim_{n\rightarrow\infty}f^{(n)}(t)$ for all $t>0, h$ satisfies the renewal equation \eqref{renewaleqn} by the monotone convergence theorem, and $\mathcal{N}_{\gamma, 2}(h)<\infty,$ whence
$$\limsup_{t\rightarrow \infty} t^{-1}\log h(t)\leq \gamma<\infty.$$
Therefore, the uniqueness portion of Theorem \eqref{theorem2} ensures that $h\equiv f.$ Because of \eqref{supersolution}, $f(t)=h(t) = \lim_{n\rightarrow \infty}f^{(n)}(t)\leq F(t)$ for all $t>0.$ This proves the assertion for $F$.

The claim about $H$ is proved similarly, but we begin our Picard iteration with $f^{(0)} := H(t)$ instead of $F(t),$ and then notice that
$$H(t) = f^{(0)}(t)\leq f^{(1)}(t)\leq f^{(2)}(t)\cdots,$$
in place of \eqref{supersolution}. The rest of the proof is the same as the one for $F.$
\end{proof}

\subsection{Proof of Theorem \ref{intermittency-thm}}
We complete the proof of Theorem \ref{intermittency-thm} in this section.
\begin{proof}
\begin{eqnarray}\label{supersolution}
\E(|u_t(x)|^2) &=& |(G_t\ast u_0)(x)|^2 + \int_0^tds\int_{{\R^d}}dy[G_{t-s}(y-x)]^2E(\sigma^2(u_s(y))\nonumber\\
&\geq&\mbox{inf}_{\substack{z\in {\R^d}}}|u_0(z)|^2 + L_\sigma^2\int_0^tI(s)\int_{{\R^d}}[G_{t-s}(y-x)]^2 dy\\
&=&\mbox{inf}_{\substack{z\in \R}}|u_0(z)|^2 +C^\ast L_\sigma^2\int_0^t\frac{I(s)}{(t-s)^{\frac{\beta d}{\alpha}}} dy\ \ \ \ \ \mbox{for all }\ t > 0,\nonumber
\end{eqnarray}
where
\begin{equation}
I(s) := \mbox{inf}_{y\in {\R^d}}\E(|u_s(y)|^2),
\end{equation}
and we find that
\begin{equation}\label{Isupersolution}
I(t) \geq \mbox{inf}_{\substack{z\in {\R^d}}}|u_0(z)|^2 + C^\ast L_\sigma^2\int_0^t\frac{I(s)}{(t-s)^{\frac{\beta d}{\alpha}}} dy\ \ \ \ \ \mbox{for all }\ t > 0.
\end{equation}
That is, $I$ is a supersolution to the renewal equation
$$f(t) = a + \int_0^tf(s)g(t-s)ds\ \ \ \ \ \ \ \ \mbox{for all }\ t > 0,$$
where $a := \mbox{inf}_{\substack{z\in {\R^d}}}|u_0(z)|^2$ and $g(t):= b/t^{\frac{\beta d}{\alpha}}$ for $b := c_0L_\sigma^2.$

Let us observe that according to \eqref{Isupersolution}, the function $I$ is a supersolution to the renewal equation \eqref{renewaleqn} with $a := \inf_{z\in {\R^d}} |u_0(z)|^2, \theta = \beta d/\alpha$ and $b := C^\ast L_\sigma^2.$ Therefore, it follows from Theorem \ref{theorem2} and \ref{sup-sub-solution} that
\begin{eqnarray}
\liminf_{t\rightarrow\infty}e^{-ct}\inf_{x\in {\R^d}}\E(|u_t(x)|^2)&=&\liminf_{t\rightarrow\infty}e^{-ct}I(t)\nonumber\\
&\geq& \liminf_{t\rightarrow\infty}e^{-ct}f(t)\nonumber\\
&=& \frac{1}{1-\gamma}\inf_{z\in{\R^d}} |u_0(z)|^2,
\end{eqnarray}
where $c = \left(b\Gamma(1-\theta)\right)^{1/(1-\theta)}.$
Take logarithms of both sides, divide by $t,$ and then finally let $t\rightarrow \infty$ to conclude the proof of Theorem \ref{intermittency-thm}.
\end{proof}

\section{Intermittency fronts}\label{intermitteny-fronts}

 Here we state and prove our second main result on the intermittency fronts for the solution of equation \eqref{tfspde}. In this section we use the notation $u_t(x):=u(t,x)$ to make the presentation concise.

  Assume that $\sigma(\cdot)$ in \eqref{tfspde} satisfies the following global Lipschitz condition, i.e. there exists a generic positive constant $\mbox{Lip}_\sigma$ such that and growth conditions:
\begin{equation}\label{Eq:Cond_a}
|\sigma(x)-\sigma(y)|\le \mbox{Lip}_\sigma|x-y|\quad\mbox{for all }\,\,x,\,y\in\R.
\end{equation}
Clearly, (\ref{Eq:Cond_a}) implies the uniform linear growth condition of $\sigma(\cdot).$ Recall the definition of $\mathscr{L}(\theta)$ from \eqref{intermittecy-front-L}.

\begin{tm}\label{maintheoremtwo}
Suppose that $d=1, \alpha=2$, measurable initial function $u_0:\RR{R}\to\RR{R}_+$  is bounded, has compact support, and is strictly positive in an open subinterval of $(0,\infty)$, and  $\sigma$  satisfies
 $\sigma(0)=0$. Then  the time fractional stochastic heat equation \eqref{tfspde} has a positive intermittency lower front. In fact,
\begin{equation}\label{intermittencypart1}
\mathscr{L}(\theta) < 0\ \ \mbox{if}\ \theta > \frac{2^{1/\beta}(\mbox{Lip}_\sigma c_0)^{4/(2-\beta)}}{(\mbox{Lip}_\sigma c_0)^{2\beta/(2-\beta)}}.
\end{equation}
In addition, under the cone condition $L_\sigma>0-$where $L_\sigma$ was defined in \eqref{conecondition}-there exists $\theta_0>0$ such that
\begin{equation}\label{intermittencypart2}
\mathscr{L}(\theta) > 0\ \ \mbox{if}\ \theta \in (0,\theta_0).
\end{equation}
That is, in this case, the stochastic heat equation has a finite intermittency upper front.
\end{tm}
This theorem in the case of the stochastic heat equation was proved by Conus and Khoshnevisan \cite{conus-khoshnevisan}.

We first prove the next proposition that implies that the solution of equation \eqref{tfspde} is square integrable over time in the language of partial differential equations.
\begin{prop}Assume that $\alpha\in (0,2]$,   and $d<\min\{2,\beta^{-1}\}\a$, then $u_t\in L^2(\R)$ a.s. for all $t\geq 0;$ in
fact, for any fixed $\epsilon\in (0,1)$ and $t\geq 0,$
\begin{equation}\label{solutionL2integrability}
\E\left(||u_t||^2_{L^2(\R)}\right)\leq \epsilon^{-1}||u_0||^2_{L^2(\R)}\exp\left(\bigg[\frac{C^\ast\Gamma(1-\beta d/\a)\mbox{Lip}_\sigma^2}{1-\epsilon}\bigg]^{\frac{1}{1-\beta d/\a}} t\right)
\end{equation}
\end{prop}
\begin{proof}
Using $|\sigma(z)\leq \mbox{Lip}_\sigma|z|$ for all $z\in\R^d$ we have,
\begin{eqnarray}
&&\E\left(|u_t^{(n+1)}(x)|^2\right)\\
&=&|(G_t\ast u_0)(x)|^2 + \int_0^t ds\int_{\R^d}dy [G_{t-s}(y-x)]^2 \E\left(|\sigma(u_s^{(n)}(y))|^2\right)\nonumber\\
&\leq&|(G_t\ast u_0)(x)|^2 + \mbox{Lip}_\sigma^2\int_0^t ds\int_{\R^d}dy [G_{t-s}(y-x)]^2 \E\left(|u_s^{(n)}(y)|^2\right).\nonumber
\end{eqnarray}
Because $\int_{\R^d}[G_{t-s}(y-x)]^2\, dy = C^\ast(t-s)^{-\beta d/\a},$ we integrate the preceding $[dx]$ and apply Fubini's theorem to see that
\begin{equation}
J^{(l)}(\gamma) := \sup_{t\geq0}\bigg[e^{-\gamma t}\E\left(||u_t^{(l)}||^2_{L^2({\R^d})}\right)\bigg]\ \ \ \ \ \ (\beta\geq 0),\nonumber
\end{equation}
solves
\begin{eqnarray}
J^{(n+1)}(\gamma)&\leq& K(\gamma) + c_1\mbox{Lip}_\sigma^2J^{(n)}(\gamma)\int_0^{t}(t-s)^{-\beta d/\a}e^{-\gamma(t - s)}ds \ \ (\beta\geq 0)\nonumber\\
&\leq & K(\gamma) + c_1\mbox{Lip}_\sigma^2J^{(n)}(\gamma)\int_0^{\infty}s^{-\beta d/\a}e^{-\gamma s}ds \nonumber\\
&=&\sup_{t\geq 0}\bigg[e^{-\gamma t}||G_t\ast u_0||^2_{L^2(\R)}\bigg] + \frac{c_1\Gamma(1-\beta d/\a)\mbox{Lip}_\sigma^2J^{(n)}(\gamma)}{\gamma^{1-\beta d/\a}}.\nonumber
\end{eqnarray}
Hence, by Cauchy-Schwarz inequality-we can see that $||G_t\ast u_0||^2_{L^2({\R^d})}\leq ||u_0||^2_{L^2({\R^d})},$ using this we obtain the following recursive inequality:
\begin{equation}
J^{(n+1)}(\gamma)\leq ||u_0||^2_{L^2({\R^d})} + \frac{c_1\Gamma(1-\beta d/\a)\mbox{Lip}_\sigma^2}{\gamma^{1-\beta d/\a}}\cdot J^{(n)}(\gamma).
\end{equation}

The preceding holds for all $n\geq 0$ and $\gamma >0.$ So we choose
$$\gamma := \gamma_\ast := \bigg[\frac{c_1\Gamma(1-\beta d/\a)\mbox{Lip}_\sigma^2}{1-\epsilon}\bigg]^{\frac{1}{1-\beta d/\a}}.$$
Then,
\begin{eqnarray}
J^{(n+1)}(\gamma_\ast)&\leq& ||u_0||^2_{L^2({\R^d})} + (1-\epsilon)J^{(n)}(\gamma_\ast)\nonumber\\
&\leq&||u_0||^2_{L^2({\R^d})} +(1-\epsilon)||u_0||^2_{L^2({\R^d})} +(1-\epsilon)^2J^{(n-1)}(\gamma_\ast)\nonumber\\
&\leq&\cdots \leq ||u_0||^2_{L^2({\R^d})}\cdot \sum_{i=0}^n (1-\epsilon)^i + J^{(0)}(\gamma_\ast)\cdot(1-\epsilon)^{n+1}\nonumber\\
&\leq&\epsilon^{-1} ||u_0||^2_{L^2({\R^d})}.\nonumber
\end{eqnarray}
since $u_t^{(n+1)}(x)$ converges to $u_t(x)$ in $L^2(\Omega)$ as $n\rightarrow \infty$ by Fatou's lemma we get
\begin{equation}
\sup_{t\geq 0}\left(e^{-\gamma_{\ast} t} \E\left(||u_t||^2_{L^2({\R^d})}\right)\right)\leq \epsilon^{-1}||u_0||^2_{L^2({\R^d})}\nonumber.
\end{equation}
This completes the proof.
\end{proof}

The proof of Theorem \ref{maintheoremtwo} requires the following ``weighted young inequality'' which is an extension of Proposition  8.3 in \cite{khoshnevisan-cbms}.

\begin{prop}\label{propostionforInt}
Let $\alpha=2$, and $d=1$.
Define for all $\gamma > 0, c\in\R,$ and $\Phi\in\mathcal{L}^{\beta,2},$
\begin{equation*}
\mathcal{N}_{\gamma,c}(\Phi):= \sup_{t\geq 0}\sup_{x\in{\R}}\bigg[e^{-\gamma t + cx}\E\left(|\Phi_t(x)|^2\right)\bigg]^{1/2}.
\end{equation*}
Then,
\begin{equation*}
\mathcal{N}_{\gamma,c}(G\circledast \Phi)\leq C(c,\gamma,\beta)\mathcal{N}_{\gamma,c}(\Phi)\ \ \ \mbox{for all}\ \gamma^\beta >\nu  c^2,
\end{equation*}
where $C(c,\gamma,\beta)$ is a finite constant that depends on $c,\gamma,$ and $\beta$.
\end{prop}

\begin{proof}
Our proof is adapted from the proof of Proposition 8.3 in \cite{khoshnevisan-cbms} with many crucial changes.
We apply the Walsh isometry in order to see that
\begin{eqnarray}
&&e^{-\gamma t + cx}\E(|(G\circledast \Phi)_t(x)|^2)\nonumber\\
&=&e^{-\gamma t + cx}\int_0^tds\int_{\R}dy [G_{t-s}(y-x)]^2\E(|\Phi_s(y)|^2)\nonumber\\
&\leq&[\mathcal{N}_{\gamma,c}(\Phi)]^2\int_0^tds\int_{\R}dy e^{-\gamma(t-s)+c(x-y)}[G_{t-s}(y-x)]^2dy\nonumber\\
&\leq&[\mathcal{N}_{\gamma,c}(\Phi)]^2\int_0^tds\int_{\R}dy e^{-\gamma s-cy}[G_s(y)]^2dy\nonumber\\
&\leq&[\mathcal{N}_{\gamma,c}(\Phi)]^2\int_0^t e^{-\gamma s}ds\int_{\R} e^{-cy}[G_s(y)]^2dy\nonumber
\end{eqnarray}
Using
\begin{eqnarray}
[G_s(y)]^2 &=& \int_0^\infty p_u(y)f_{E_s}(u)du\int_0^\infty p_v(y)f_{E_s}(v)dv\\ \nonumber
&=&\int_0^\infty\int_0^\infty p_u(y)p_v(y)f_{E_s}(u)f_{E_s}(v)dudv
\end{eqnarray}
where $p_u(y) = \frac{e^{-\frac{|y|^2}{4\nu u}}}{\sqrt{4\nu \pi u}}$, $p_u(y)\leq \frac{1}{\sqrt{4\nu\pi u}}$ and a standard moment-generating computation for Gaussian laws we have,
\begin{eqnarray}
&&\int_{\R} e^{-cy}[G_s(y)]^2 dy \nonumber\\
&\leq&\int_0^\infty\int_0^\infty\frac{f_{E_s}(u)du}{\sqrt{4\pi\nu  u}}\int_{\R}e^{-cy}p_v(y)dyf_{E_s}(v)dv\nonumber\\
&=&\int_0^\infty\frac{f_{E_s}(u)du}{\sqrt{4\pi\nu  u}}\int_{0}^\infty e^{\nu c^2v}f_{E_s}(v)dv.\nonumber
\end{eqnarray}
Using uniqueness of Laplace transform we can easily show $\E[E_s^k] = \frac{\Gamma(1+k)s^{\beta k}}{\Gamma(1+\beta k)}$ for $k>-1$ and using these moments we can also compute the moment-generating function of inverse stable subordinator $E_s$ as
$$\E[e^{wE_s}] = \sum_{k=0}^{\infty}\frac{w^k\E(E_s^k)}{k!}=\sum_{k=0}^{\infty}\frac{w^ks^{\beta k}}{\Gamma(1+\beta k)}= E_\beta(ws^\beta).$$
Hence, we get
\begin{eqnarray}
\int_{\R} e^{-cy}[G_s(y)]^2 dy \leq \int_0^\infty E_\beta(\nu c^2s^\beta)\frac{f_{E_s}(u)du}{\sqrt{4\pi\nu u}}\leq\frac{s^{-\beta/2}E_\beta(\nu c^2s^\beta)}{2\sqrt{\nu}\Gamma(1-\beta/2)}\nonumber.
\end{eqnarray}
Therefore, using D. Kershaw inequality $\Gamma(x+\lambda)/\Gamma(x+1)<1/(x+1/2)^{1-\lambda}$ for $0<\lambda <1, x>0$ we have
\begin{eqnarray}
&&e^{-\gamma t + cx}\E(|(G\circledast \Phi)_t(x)|^2)\nonumber\\
&\leq&\frac{[\mathcal{N}_{\gamma,c}(\Phi)]^2}{2\sqrt{\nu}\Gamma(1-\beta/2)}\sum_{k=0}^{\infty}\frac{\nu^kc^{2k}}{\Gamma(1+\beta k)}\int_0^\infty e^{-\gamma s}s^{\beta k-\beta/2}ds\nonumber\\
&=&\frac{[\mathcal{N}_{\gamma,c}(\Phi)]^2}{2\sqrt{\nu}\Gamma(1-\beta/2)\gamma^{1-\beta/2}}\sum_{k=0}^{\infty}\frac{\nu^kc^{2k}}{\gamma^{\beta k}}\frac{\Gamma(1+\beta k-\beta/2)}{\Gamma(1+\beta k)}\nonumber\\
&\leq&\mathcal{N}_{\gamma,c}(\Phi)]^2\frac{2^{\beta/2-1}}{\sqrt{\nu}\Gamma(1-\beta/2)\gamma^{1-\beta/2}}\sum_{k=0}^{\infty}\frac{\nu^kc^{2k}}{\gamma^{\beta k}}\nonumber,
\end{eqnarray}
the last series converges for $\gamma^\beta > \nu c^2.$ The right-hand side is independent of $(t,x).$ Therefore, we optimize over $(t,x)$ and then take square roots of both sides in order to finish the proof.
\end{proof}
The next corollary is an extension of Corollary 8.4 in \cite{khoshnevisan-cbms}.
\begin{coro}\label{coro3}
If $c^{1/\beta-1/2}>\mbox{Lip}_\sigma\sqrt{\frac{2^{\beta/2+1/2-1/\beta}}{\nu^{1/\beta}\Gamma(1-\beta/2)}}$, then the solution to the tfspde \eqref{tfspde} for  $\a=2$ satisfies
\begin{equation}
\E(|u_t(x)|^2)\leq A(c,\beta)\exp\left(-c|x| + (2\nu c^2)^{1/\beta}t\right),
\end{equation}
simultaneously for all $x\in\R$ and $t\geq 0,$ where $A(c,\beta)$ is a finite constant that depends only on $c$ and $\beta.$
\end{coro}

\begin{proof}The proof is adapted form the proof of Corollary 8.4 in \cite{khoshnevisan-cbms} with many crucial changes.
Recall that for all $\gamma > 0$
\begin{eqnarray}
\mathcal{N}_{\gamma,c}(u^{(n+1)})&\leq& [\mathcal{N}_{\gamma,c}(G_t\ast u_0)] + [\mathcal{N}_{\gamma,c}(G\circledast \sigma(u^{(n)})]\nonumber\\
&\leq&[\mathcal{N}_{\gamma,c}(G_t\ast u_0)] + C(c,\gamma,\beta)[\mathcal{N}_{\gamma,c}(\sigma(u^{(n)}))],\nonumber
\end{eqnarray}
using Proposition \ref{propostionforInt}. Because $\sigma(z)\leq \mbox{Lip}_\sigma|z|$ for all $z\in\R,$
$$\mathcal{N}_{\gamma,c}(\sigma(u^{(n)})\leq \mbox{Lip}_\sigma\mathcal{N}_{\gamma,c}(u^{(n)}).$$
Also,
\begin{eqnarray}
e^{-\gamma t + cx}(G_t\ast |u_0|)(x)&=&e^{-\gamma t}\int_{\R}G_t(y-x)e^{-c(y-x)}e^{cy}|u_0(y)|dy\nonumber\\
&\leq& e^{-\gamma t}\mathcal{N}_{0,c}(u_0)\int_{\R}e^{cz}G_t(z)dz\nonumber\\
&=& e^{-\gamma t}E_\beta(\nu c^2t^\beta)\mathcal{N}_{0,c}(u_0).
\end{eqnarray}
We take $\gamma^\beta := 2\nu c^2$ to see that for all integers $k\geq 0$
$$e^{-\gamma t}E_\beta(\nu c^2t^\beta) = \sum_{k=0}^\infty \frac{\nu^k c^{2k}t^{\beta k}e^{-\gamma t}}{\Gamma(1+\beta k)}=\sum_{k=0}^\infty \frac{\nu^k c^{2k}}{\gamma^{\beta k}}\frac{u^{\beta k}e^{-u}}{\Gamma(1+\beta k)}\leq 2,$$
since $\frac{u^{\beta k}e^{-u}}{\Gamma(1+\beta k)}<1.$ In this case
$$C(c,\gamma,\beta)=\sqrt{\frac{2^{\beta/2}}{\sqrt{\nu}\Gamma(1-\beta/2)((2\nu c^2)^{1/\beta})^{1-\beta/2}}}=\sqrt{\frac{2^{\beta/2+1/2-1/\beta}}{\nu^{1/\beta}\Gamma(1-\beta/2)c^{2/\beta-1}}},$$
and $$C(c,\gamma,\beta)\mbox{Lip}_\sigma<1.$$
We see that for  all integers $n\geq 0$,
$$\mathcal{N}_{(2\nu c^2)^{1/\beta},c}(u^{(n+1)})\leq 2\mathcal{N}_{0,c}(u^{n+1}) + C(c,\gamma,\beta)\mbox{Lip}_\sigma\mathcal{N}_{(2\nu c^2)^{1/\beta},c}(u^{(n)}).$$
Similarly,
$$\mathcal{N}_{(2\nu c^2)^{1/\beta},-c}(u^{(n+1)})\leq 2\mathcal{N}_{0,-c}(u^{(n+1)}) + C(c,\gamma,\beta)\mbox{Lip}_\sigma\mathcal{N}_{(2\nu c^2)^{1/\beta},-c}(u^{(n)}).$$
Since $u_0$ has compact support, it follows that $\mathcal{N}_{0,c}(u_0)+\mathcal{N}_{0,-c}(u_0)<\infty.$ Therefore, it follows readily from the preceding discussion that, because $C(c, (2\nu c^2)^{1/\beta}, \beta)\mbox{Lip}_\sigma<1,$
$$\sup_{n\geq 0}\left[\mathcal{N}_{(2\nu c^2)^{1/\beta},c}(u^{(n+1)})+\mathcal{N}_{(2\nu c^2)^{1/\beta},-c}(u^{(n+1)})\right]<\infty.$$
Since $u_t^{(n+1)}(x)$ converges to $u_t(x)$ in $L^2(\Omega)$ as $n\rightarrow \infty$  Fatou's lemma implies that
$$\mathcal{N}_{(2\nu c^2)^{1/\beta},c}(u)+\mathcal{N}_{(2\nu c^2)^{1/\beta},-c}(u)<\infty.$$
The corollary follows readily from this fact.
\end{proof}

We are ready to prove Theorem \ref{maintheoremtwo}. We do this  in two steps adapting the method in \cite[Chapter 8]{khoshnevisan-cbms} with crucial nontrivial changes: First we derive \eqref{intermittencypart1}; and then we establish \eqref{intermittencypart2}.

\begin{proof}[{\bf Proof of  \eqref{intermittencypart1}}]
 Since $u_0$ has compact support, it follows that $|u_0(x)|= O(e^{c|x|})$ for all $c>0.$ Therefore, we may apply Corollary \ref{coro3} to an arbitrary $c^{1/\beta-1/2}>\mbox{Lip}_\sigma\sqrt{\frac{2^{\beta/2+1/2-1/\beta}}{\nu^{1/\beta}\Gamma(1-\beta/2)}}:= \mbox{Lip}_\sigma c_0$ in order to see that
\begin{eqnarray}
\mathscr{L}(\theta)&=& \limsup_{t\rightarrow\infty}\frac{1}{t}\sup_{|x|>\theta t}\log \E\left(|u_t(x)|^2\right)\leq - \sup_{c>(\mbox{Lip}_\sigma c_0)^{2\beta/(2-\beta)}}\bigg[\theta c - (2c^2)^{1/\beta}\bigg]\nonumber\\
&\leq&-\bigg[\theta (\mbox{Lip}_\sigma c_0)^{2\beta/(2-\beta)}) - 2^{1/\beta}(\mbox{Lip}_\sigma c_0)^{4/(2-\beta)})\bigg],
\end{eqnarray}
obtained by setting $c:= (\mbox{Lip}_\sigma c_0)^{2\beta/(2-\beta)}$ in the maximization problem of the first line of preceding display. The right-most quantity is strictly negative when $$\theta > \frac{2^{1/\beta}(\mbox{Lip}_\sigma c_0)^{4/(2-\beta)})}{(\mbox{Lip}_\sigma c_0)^{2\beta/(2-\beta)})};$$ this proves\eqref{intermittencypart1}.
\end{proof}

\begin{proof}[{\bf Proof of  \eqref{intermittencypart2}.}] According to \eqref{supersolution}
\begin{eqnarray}\label{intermittencyproof}
&&\E(|u_t(x)|^2)\nonumber\\
&\geq&|(G_t\ast u_0)(x)|^2 + L_\sigma^2\int_0^tds\int_{\R}dy[G_{t-s}(y-x)]^2\E(|u_s(y)|^2),
\end{eqnarray}
for all $t>0$ and $x\in\R.$ Also, note that if $x,y\in\R, 0\leq s\leq t,$ and $\a\geq 0,$ then
$$1_{[\theta t,\infty)}(x)\geq 1_{[\theta(t-s),\infty)}(x-y)\cdot 1_{[\theta s,\infty)}(y).$$
This is a consequence of the triangle inequality. Therefore,
\begin{eqnarray}
&&\int_{\theta t}^\infty\int_0^tds\int_{\R}dy[G_{t-s}(y-x)]^2\E(|u_s(y)|^2)\nonumber\\
&\geq& \int_0^tds\left(\int_{\theta(t-s)}^\infty[G_{t-s}(z)]^2dz\right)\left(\int_{\theta s}^\infty \E(u_s(y)|^2)dy\right).\end{eqnarray}
This and \eqref{intermittencyproof} together show that the function
\begin{equation}
M_+(t) := \int_{\theta t}^\infty \E(|u_s(y)|^2)dy
\end{equation}
satisfies the following renewal inequality:
\begin{equation}\label{positivepart}
M_+(t)\geq \int_{\theta t}^\infty|(G_t\ast u_0)(x)|^2dx + L_\sigma^2(T\ast M_+)(t),
\end{equation}
with $$T(t) := \int_{\theta t}^\infty[G_t(z)]^2dz.$$
Because of symmetry we can write $T(t) = \int^{-\theta t}_{-\infty}[G_t(z)]^2dz.$ Therefore, a similar argument shows that the function
$$M_{-}(t) := \int_{-\infty}^{-\theta t}\E(|u_s(y)|^2)dy,$$
satisfies the following renewal inequality:
\begin{equation}\label{negativepart}
M_{-}(t)\geq \int^{-\theta t}_{-\infty}|(G_t\ast u_0)(x)|^2dx + L_\sigma^2(T\ast M_{-})(t).
\end{equation}
Define
$$M(t) := \int_{|y|>\theta t} \E(|u_t(y)|^2)dy = M_+(t) + M_{-}(t),$$
in order to deduce from \eqref{positivepart} and \eqref{negativepart} that
$$M(t)\geq \int_{|x|>\theta t}|(G_t\ast u_0)(x)|^2dx + L_\sigma^2(T\ast M)(t).$$
Define $\mathcal{L}\phi$ to be the Laplace transform of any measurable function $\phi:\R+\rightarrow\R_+.$ That is,
$$(\mathcal{L}\phi)(\lambda) = \int_0^\infty e^{-\lambda t}\phi(t)dt\ \ \ (\lambda \geq 0).$$
Then, we have the following inequality of Laplace transforms: For every $\lambda\geq 0,$
\begin{eqnarray}\label{laplacetransform}
&&(\mathcal{L}M)(\lambda)\\
&\geq& 2\int_0^\infty e^{-\lambda t}dt\int_{\theta t}^\infty dx |(G_t\ast u_0)(x)|^2 + L_\sigma^2 (\mathcal{L}T)(\lambda)(\mathcal{L}M)(\lambda).\nonumber
\end{eqnarray}
Since $$(\mathcal{L}T)(0) = \int_0^\infty\frac{dt}{t}\int_{\theta t}^\infty[G_t(z)]^2dz.$$
Therefore, there exists $\theta_0>0$ such that $(\mathcal{L}T)(0)>L_\sigma^{-2}$ whenever $\theta\in(0,\theta_0).$ This and dominated convergence theorem together imply that there, in turn, will exist $\lambda_0>0$ such that $(\mathcal{L}T)(\lambda)> L_\sigma^{-2}$ whenever $\theta\in(0,\theta_0)$ and $\lambda\in(0,\lambda_0).$ Since $u_0>0$ on a set of positive measure, it follows readily that $$\int_0^\infty e^{-\lambda t}dt\int_{\theta t}^\infty dx |(G_t\ast u_0)(x)|^2 > 0,$$
for all $\theta, \lambda\geq 0,$ including $\theta\in(0,\theta_0)$ and $\lambda\in(0,\lambda_0).$ Therefore, \eqref{laplacetransform} implies that
\begin{equation}
(\mathcal{L}M)(\lambda) = \infty\ \ \ \mbox{for}\ \theta\in(0,\theta_0)\ \mbox{and}\ \lambda\in(0,\lambda_0).
\end{equation}
One can deduce from this and the definition of $M$ that
$$\limsup_{t\rightarrow\infty}e^{-\lambda t}\int_{|y|>\theta t}\E(|u_t(y)|^2)dy = \infty,$$
whenever $\theta\in(0,\theta_0)\ \mbox{and}\ \lambda\in(0,\lambda_0).$ This and the already-proven first part \eqref{intermittencypart1} together show that
$$\limsup_{t\rightarrow\infty}e^{-\lambda t}\int_{\theta t <|y|< \gamma t}\E(|u_t(y)|^2)dy = \infty,$$
whenever $\theta\in(0,\theta_0), \lambda\in(0,\lambda_0) \ \mbox{and}\ \gamma > \frac{2^{1/\beta}(\mbox{Lip}_\sigma c_0)^{4/(2-\beta)})}{(\mbox{Lip}_\sigma c_0)^{2\beta/(2-\beta)})}.$ Since the last integral is not greater than $(\gamma - \theta)t\sup_{|x|>\theta t}\E(|u_t(x)|^2),$ it follows that
$$\mathscr{L}(\theta)=\limsup_{t\rightarrow\infty}\frac{1}{t}\sup_{|x|>\theta t}\log \E(|u_t(x)|^2) \geq \lambda_0,$$
for $\theta\in(0,\theta_0).$ This proves \eqref{intermittencypart2} and hence the theorem.

\end{proof}

\end{document}